\documentclass[bj]{imsart}

\RequirePackage[OT1]{fontenc}
\RequirePackage{amsthm,amsmath}
\RequirePackage[numbers,sort&compress]{natbib}
\RequirePackage[colorlinks,citecolor=blue,urlcolor=blue]{hyperref}

\arxiv{1512.00209}

\usepackage[USenglish]{babel}
\usepackage{csquotes}
\usepackage{amsfonts,amssymb}
\usepackage{bbm}
\usepackage{mathtools}
\usepackage{enumerate}
\usepackage[dvipsnames]{xcolor}
\usepackage{tikz}
\usepackage{graphicx}
\usepackage{subfig}

\usepackage[capitalise]{cleveref}
\crefformat{equation}{(#2#1#3)}
\crefrangeformat{equation}{(#3#1#4) to~(#5#2#6)}
\crefmultiformat{equation}{(#2#1#3)}%
{ and~(#2#1#3)}{, (#2#1#3)}{ and~(#2#1#3)}
\makeatletter
\if@cref@capitalise
\crefname{bsp}{Example}{Examples}
\crefrangeformat{bsp}{Examples #3#1#4 to~#5#2#6}
\crefmultiformat{bsp}{Examples #2#1#3}%
{ and~#2#1#3}{, #2#1#3}{ and~#2#1#3}
\crefname{prop}{Proposition}{Propositions}
\crefrangeformat{prop}{Propositions #3#1#4 to~#5#2#6}
\crefmultiformat{prop}{Propositions #2#1#3}%
{ and~#2#1#3}{, #2#1#3}{ and~#2#1#3}
\else
\crefname{bsp}{example}{examples}
\crefrangeformat{bsp}{examples #3#1#4 to~#5#2#6}
\crefmultiformat{bsp}{examples #2#1#3}%
{ and~#2#1#3}{, #2#1#3}{ and~#2#1#3}
\crefname{prop}{proposition}{propositions}
\crefrangeformat{prop}{propositions #3#1#4 to~#5#2#6}
\crefmultiformat{prop}{propositions #2#1#3}%
{ and~#2#1#3}{, #2#1#3}{ and~#2#1#3}
\fi
\makeatother
\Crefname{bsp}{Example}{Examples}
\Crefname{prop}{Proposition}{Propositions}

\startlocaldefs
\numberwithin{equation}{section}
\theoremstyle{plain}
\newtheorem{thm}{Theorem}
\newtheorem{lem}{Lemma}
\newtheorem{prop}{Proposition}
\newtheorem{mydef}{Definition}
\newtheorem{bsp}{Example}
\newtheorem{rem}{Remark}
\endlocaldefs

\newcommand{\ls}[1]{{\scriptsize{#1}}}
\newcommand{\lss}[1]{{\tiny{#1}}}
\newcommand{\kreis}{\tikz{\node[shape=circle,draw,inner sep=1.5pt] {};}}
\newcommand{\vertex}[1]{\tikz[baseline=(char.base)]{
            \node[shape=circle,draw,inner sep=0pt] (char) {$v_{#1}$};}}
\newcommand{\stage}[2]{\tikz[baseline=(char.base)]{
            \node[shape=circle,draw,inner sep=0pt,fill={#1}] (char) {$v_{#2}$};}}
\newcommand{\csep}{1}
\newcommand{\rsep}{1}
\usepackage{fix-cm}
\DeclareMathOperator{\ind}{\mathbbm{1}}
\DeclareMathOperator{\pa}{pa}
\DeclareMathOperator{\ch}{ch}
\newcommand{\N}{\mathbb{N}}
\newcommand{\R}{\mathbb{R}}
\newcommand{\F}{{\mathcal{F}}}
\newcommand{\T}{{\mathcal{T}}}
\newcommand{\St}{{\mathcal{S}}}
\newcommand{\ttheta}{\boldsymbol{\theta}}
\newcommand{\pt}{(\T,\ttheta_{\T})}
\newcommand{\ptrep}{[\T,\ttheta_{\T}]}
\newcommand{\st}{(\St,\ttheta_{\St})}
\newcommand{\ft}{(\F,\ttheta_{\F})}
\newcommand{\treemodel}{\mathbb{P}_{\pt}}
\newcommand{\tree}{\T=(V,E)}
\newcommand{\ppitheta}{\boldsymbol{\pi}_{\ttheta,\T}}
\newcommand{\pitheta}{\pi_{\ttheta,\T}}
\newcommand{\Ptheta}{P_{\ttheta}}
\newcommand{\prodm}[1]{\prod_{\mathclap{#1}}}
\newcommand{\summ}[1]{\sum_{\mathclap{#1}}}
\newcommand{\fact}{s(c(\ttheta))}
\newcommand{\iotat}{\iota_{\T}}
\newcommand{\emap}{\mathfrak{c}}
\newcommand{\smap}{\mathfrak{s}}
\newcommand{\fmap}{\mathfrak{r}}

\begin{document}

\begin{frontmatter}
\title{Equivalence Classes of Staged Trees}
\runtitle{Equivalence Classes of Staged Trees}
\begin{aug}
\author{\fnms{Christiane} \snm{G\"orgen}\ead[label=e1]{c.gorgen@warwick.ac.uk}}\and
\author{\fnms{Jim Q.} \snm{Smith}\ead[label=e2]{j.q.smith@warwick.ac.uk}}
\runauthor{Ch.~G\"orgen and J.~Q.~Smith}
\affiliation{University of Warwick}

\address{Department of Statistics\\
University of Warwick\\
Coventry CV4 7AL\\
United Kingdom\\
\printead{e1,e2}}
\end{aug}

\def\abstractname{ }
\begin{abstract}
In this paper we give a complete characterization of the statistical equivalence classes of CEGs and of staged trees. We are able to show that all graphical representations of the same model share a common polynomial description. Then, simple transformations on that polynomial enable us to traverse the corresponding class of graphs. We illustrate our results with a real analysis of the implicit dependence relationships within a previously studied dataset.
\end{abstract}

\begin{keyword}[class=MSC]
\kwd[Primary ]{60E05}
\kwd{60K35}
\kwd[; secondary ]{62E99}
\end{keyword}

\begin{keyword}
\kwd{Algebraic Statistics}
\kwd{Chain Event Graphs}
\kwd{Probability Trees}
\kwd{Staged Trees}
\end{keyword}
\end{frontmatter}

\section{Introduction}

The \emph{Chain Event Graph (CEG)} is a discrete statistical model based on a graphical description given by an event tree~\cite{smithanderson:ceg}. CEGs have now successfully led statistical inference in a whole range of domains~\cite{lorna:bnCEG, cowelljim:causal,freeman:MAP,thwaitessmithcowell:propagation}. However, a formal analysis of the statistical properties of this class of models is long overdue.

In this paper, it will be most convenient to represent a CEG model by a corresponding \emph{staged tree}~\cite{smithanderson:ceg}. 
From this colored graph we can read a  parametrization rule given by the multiplication of transition probabilities along root-to-leaf paths. Two staged trees are said to be \emph{statistically equivalent} if their  parametrization rules parametrize the same model: see \cref{sect:trees}.

The study of these statistical equivalence classes is an important one. The first reason for this is computational: CEGs constitute a massive model space to explore. By identifying a single representative within an equivalence class and a priori selecting across these representatives rather than the full class, we can dramatically reduce the search effort across this space. The second reason concerns coherence: when adopting a Bayesian approach in model selection, \cite{heckerman}~and others have argued that two statistically equivalent models (i.e.\,those always giving the same likelihood) should be given the same prior distribution over its parameters. To apply this principle, it is essential to know when two CEGs make the same distributional assertions. The third reason is inferential: just like a Bayesian network (BN), a CEG or staged tree has a natural causal extension~\cite{cowelljim:causal,thwaites:causal}. So, in particular, causal discovery algorithms can be applied to CEGs to elicit a putative causal ordering between various associated variables. A strong argument is that a necessary condition for a causal deduction to be made from a given dataset is that this deduction is invariant to the choice of one representative within a statistical equivalence class. So again we need to be able to identify equivalence classes of a hypothesized causal CEG in order to perform these algorithms.

Now, unlike for BNs, where model representations making equivalent distributional assumptions can be elegantly characterized through their sharing the same \emph{essential graph} \cite{amp:equiv,heckerman}, sadly no such common representation is available for staged trees or CEGs.
However, we show here that we can instead specify staged trees in terms of a nested polynomial representation. This then provides a natural algebraic index for a class of equivalent staged trees and an analogue of the essential graph.
Because staged tree models include discrete BN models as a special case, our polynomial characterization also gives an alternative to the ansatz adopted by~\cite{geigermeeksturm:algebra}.

Our central theorem, presented in \cref{sect:polynomials}, is based on two main findings. First, the \emph{interpolating polynomial} of a staged tree can capture certain context-specific independence structures that are invariant with respect to a class of graphical transformations we call \emph{swaps}. These transformations are analogous to arc reversals sometimes applied to BN models~\cite{schachter}. Second, by substituting various monomial terms of the interpolating polynomial into single factors we can often simplify our representation to capture only its substantive structure. Within our development this corresponds to what we call here a \emph{resize} operator on the staged tree. We show later that in the context of decomposable BNs, this operation is analogous for example to the transformation of a directed acyclic graph into a junction tree~\cite{lauritzen}. Swaps and resizes enable us to meaningfully incrementally traverse the class of statistically equivalent staged tree representations of a given model. We are able to show that between every two statistically equivalent staged trees there is a map which is a composition of these operators. Statistical equivalence classes of staged tree and CEG models are thus fully characterized through simple relationships between their interpolating polynomials.

We illustrate our methods by giving a full characterization of the statistical equivalence class and a putative causal interpretation of the staged tree representing the Christchurch Health and Development Study~\cite{chds} in \cref{sect:causal}. We end the paper with a brief discussion.

\section{Staged Tree Statistical Models}\label{sect:trees}

In this paper we study properties of parametric statistical models which are based on a graphical representation given by a probability tree \cite{shafer:causalconj,smithanderson:ceg}. We will treat the probability tree not only as some easily interpretable picture but as a directed graphical model in its own right. To properly study equivalence classes of these models, we first need to tighten the formalism introduced in~\cite{smithanderson:ceg}.
\medskip

A finite graph $\tree$ with vertex set $V$ and edge set $E\subseteq V\times V$ is called a \emph{tree} if it is connected and has no cycles \cite{shafer:causalconj}.
In a \emph{directed} tree, each edge $e=(v,v')\in E$ is a pair of ordered vertices. We call vertices $\pa(v)=\{v'~|~\text{there is } (v',v)\in E\}$ the \emph{parents} of $v\in V$ and $\ch(v)=\{v'\in V~|~\text{there is } (v,v')\in E\}$ the set of \emph{children} of a vertex $v\in V$. A vertex $v_0\in V$ without parents is called a \emph{root} of the tree and vertices without children are called \emph{leaves}. We use the term \emph{root-to-leaf path} and the sybol $\lambda$ for a directed and connected sequence of edges $E(\lambda)\subseteq E$ which emanate from a root and terminate in a leaf. We call a directed tree an \emph{event tree} if all vertices except for one unique root have exactly one parent and each parent which is not a leaf has at least two children.

In a tree model (as defined below), every root-to-leaf path represents an atom in a given sample space and depicts one possible history of a unit in a population passing through the tree. Every vertex $v\in V$ denotes a situation that such a unit might find itself in during that progress, and every edge $e=(v,v')\in E$ denotes the possibility of passing from one situation $v$ to the next $v'$. For any unit in the population there are always at least two possible unfoldings from every situation it might pass through.

We denote the set of all root-to-leaf paths of an event tree by $\Lambda(\T)$. For fixed $v\in V$ we define a \emph{vertex-centered event} as $\Lambda(v)=\bigl\{\lambda\in\Lambda(\T)~|~\text{there is } (\cdot,v)\in E(\lambda)\bigr\}$ and set $\Lambda(v_0)=\Lambda(\T)$. In tree models, the set of all root-to-leaf paths going through one fixed vertex is the set of all atoms for which that situation happens.
We call a vertex $v\in V$ together with its emanating edges $E(v)=\{(v,v')\in E~|~v'\in\ch(v)\}$  a \emph{floret}, denoted $\F_v=(v,E(v))$. If $v$ is a leaf then $E(v)=\emptyset$ and $\F_v$ is an empty floret. If the entire event tree is a single floret, it is called a \emph{star}.

The directionality of an event tree induces a natural order on events (and florets) as follows: We say that $\Lambda(v)$ is \emph{upstream} of $\Lambda(v')$ if and only if every root-to-leaf path $\lambda\in\Lambda(v)\cap\Lambda(v')$ is a sequence of edges containing $(v,\cdot)$ before $(v',\cdot)$. Tree models are therefore particularly useful if a model class needs to express a potential ordering of events rather than of random variables~\cite{smithanderson:ceg}.

\begin{mydef}[Probability tree]\label{def:probtree}
Let $\tree$ be an event tree with parameters $\theta(e)=\theta(v,v')$ associated to all edges $e=(v,v')\in E$. We call $\ttheta_v=\bigl(\theta(e)~|~e\in E(v)\bigr)$ a \emph{vector of floret parameters}.

The pair $\pt$ of tree graph and labels $\ttheta_{\T}=(\ttheta_v~|~v\in V)$ is called a \emph{probability tree} if every floret parameter vector lies inside a probability simplex, so $\sum_{e\in E(v)}\theta(e)=1$ and $\theta(e)\in(0,1)$ for all $v\in V$, $e\in E$. In probability trees, we call each parameter $\theta(e)$, $e\in E$, a \emph{primitive probability}.
\end{mydef}

Primitive probabilities can be thought of as (conditional) transition probabilities along root-to-leaf paths. Throughout, we assume these probabilities to be strictly positive in order to avoid various distracting technical issues concerning boundary cases.

We henceforth denote the product of all primitive probabilities along a root-to-leaf path $\lambda\in\Lambda(\T)$ in a probability tree by
\begin{equation}\label{eq:primprob}
\pitheta(\lambda)=\prodm{e\in E(\lambda)}\theta(e)
\end{equation}
where $\ttheta=\ttheta_{\T}$ for short. It is straightforward to show that $\pitheta$ is a strictly positive probability mass function. In particular, atomic probabilities sum to unity because of the floret sum-to-$1$ conditions in \cref{def:probtree}.

Let $\ppitheta=\bigl(\pitheta(\lambda)~|~\lambda\in\Lambda(\T)\bigr)$ denote a vector of atomic probabilities represented by a probability tree. Following standard notation in algebraic statistics, we always think of a family of discrete probability distributions as a set of points. So let
\begin{equation}\label{eq:treemodel}
\treemodel=\Bigl\{\ppitheta~|~\ttheta\in\bigtimes_{\mathclap{v\in V}}\Delta_{\#E(v)-1}^\circ\Bigr\}\subseteq\Delta_{\#\Lambda(\T)-1}^\circ
\end{equation}
where $\Delta_{n-1}^\circ=\{\boldsymbol{p}\in\R^n~|~{\sum_{i=1}^np_i=1} \text{ and }{p_i\in(0,1)} \text{ for all }i\in[n]\}$ denotes a probability simplex, $[n]=\{1,2,\ldots,n\}$~\cite{oberwolfach}. We call the  parametric statistical model in \cref{eq:treemodel} a \emph{(probability) tree model} and say that the elements in $\treemodel$ \emph{factorize according to} $\T$. This terminology is analogous to BN models where distributions factorize according to an acyclic digraph~\cite{lauritzen}.

Henceforth, we will call two probability tree representations $\pt$ and $\st$ of the same model $\treemodel=\mathbb{P}_{\st}$ \emph{statistically equivalent}. We let the symbol $\ptrep$ denote the set of all probability tree representations of $\treemodel$.\medskip

We can always identify the set of root-to-leaf paths of a probability tree with a finite space $\Omega$ via a bijection 
\begin{equation}\label{eq:iota}
\iotat:~\Omega\to\Lambda(\T),\quad \omega\mapsto\bigl(e~|~e\in E(\iotat(\omega))\bigr)
\end{equation} 
which maps an \emph{atom} or atomic event to a sequence of edges. Importantly, $\pitheta$ then induces a measure ${\Ptheta=\pi_{\ttheta,\T}\circ\iotat}$ on $\Omega$ which does not depend on the graph $\T$. We will usually use the symbol $\Ptheta(\omega)$ to refer to a value in $(0,1)$ and $\pitheta(\iotat(\omega))$ to refer to a symbolic product of parameters, $\omega\in\Omega$. To make this distinction, we also call $\pitheta$ an \emph{atomic monomial} rather than an atomic probability. So two statistically equivalent staged trees need to have the same underlying space $\Omega$ and the same distribution $\Ptheta$ over its atoms. Using  \cref{eq:iota}, root-to-leaf paths with the same meaning (representing the same atom) can then be identified across different representations.

A tree model does not need to arise from an underlying set of problem variables. However, if it is naturally defined through the relationships between a set of pre-specified random variables then we can identify the state space of these variables with a set of root-to-leaf paths as in \cref{eq:iota}. This enables us for instance to represent a discrete BN by a probability tree.

\begin{bsp}\label{bsp:saturated}
Let $\pt$ be a probability tree with $n\in\N$ root-to-leaf paths and a probability mass function $\pitheta$ as above. The vector $\ppitheta\in\Delta_{n-1}^\circ$ can then take any value within the probability simplex and we call $\treemodel=\Delta_{n-1}^\circ$ a \emph{saturated} tree model.

Let $\F=(\{v_0\}\cup\ch(v_0),\{e_1,\ldots,e_n\})$ be a star with attached parameter vector $\ttheta_{\F}=\bigl(\theta(e_i)~|~i\in[n]\bigr)\in\Delta_{n-1}^\circ$. Then $\ft$ is statistically equivalent to $\pt$ if the probabilities associated with the same atoms are identified for any choice of parameters: so $\theta(e_i)=\pitheta(\iota_{\T}(\omega_i))$ for all $\iota_{\F}^{-1}(e_i)=\omega_i\in\Omega$ and every $i\in[n]$.
\end{bsp}

Probability trees are most interesting when two or more vectors of floret parameters take the same values, and the distributions $\pitheta$ factorize according to a \enquote{colored} graph~$\T$.

\begin{mydef}[Staged tree]\label{def:stages}
Let $\pt$ with $\tree$ be a probability tree. We define an equivalence relation which relates two vertices $v,w\in V$ if and only if their parameter vectors coincide $\ttheta_v=\ttheta_w$ up to a permutation of their components. Then $v$ and $w$ are said to be in the same \emph{stage} and $\pt$ is said to be a \emph{staged tree}. 

If no related vertices $v,w\in V$ are connected by a root-to-leaf path, $\Lambda(v)\cap\Lambda(w)=\emptyset$, we will call $\pt$ \emph{square-free}.
\end{mydef}

Whenever two vertices are in the same stage and a unit arrives at one of them, the transition probabilities to all children of that vertex will not depend on which of the two vertices the unit is actually in, and will thus not depend on the path that unit took to arrive in that situation. The transition probabilities from these stages are thus independent of upstream events. We always assign the same \emph{color} to all vertices in the same stage. In this way, all modelling assumptions in staged tree models are coded purely graphically and are very easy to communicate~\cite{lorna:bnCEG,ours}.

When having a preassigned set of random variables, setting floret parameter vectors equal to each other can be interpreted as specifying a set of \emph{context-specific} conditional independences as in \cite{boutilier}. Models with these types of constraints are now widely used in BN modelling, especially when the domain of application is large. 


\begin{bsp}\label{bsp:lorna1}
The staged tree $\pt$ depicted in \cref{fig:lorna1} is a simplified detail of the graph analyzed in~\cite{lorna:bnCEG}. Here, every atom is represented by a root-to-leaf path with two edges: the first depicts the socio-economic background of a child, the second corresponds to a number of life events. See also \cref{sect:causal} below.

Information of the type \enquote{if we know the social status of a child's family, then their number of life events does not depend on their economic situation} can be embedded graphically by collecting the vertices $v_1, v_2\in u_{\text{blue}}$ and $v_3, v_4\in u_{\text{green}}$ in the blue- and green-colored stage, respectively. The primitive probabilities on the edges of the corresponding florets $\ttheta_{v_1}=\ttheta_{v_2}$ and $\ttheta_{v_3}=\ttheta_{v_4}$ are then identified.
\end{bsp}

\begin{figure}[t]
\centering
\begin{tikzpicture}
\renewcommand{\csep}{2}
\renewcommand{\rsep}{0.34}
\node (v0) at (0*\csep,4*\rsep) {\vertex{0}};
\node (v1) at (1*\csep,6*\rsep) {\stage{SkyBlue}{1}};
\node (v2) at (1*\csep,4.5*\rsep) {\stage{SkyBlue}{2}};
\node (v3) at (1*\csep,2.5*\rsep) {\stage{Green!70}{3}};
\node (v4) at (1*\csep,1*\rsep) {\stage{Green!70}{4}};
\node (v5) at (2*\csep,7*\rsep) {\kreis};
\node (v6) at (2*\csep,6*\rsep) {\kreis};
\node (v7) at (2*\csep,5*\rsep) {\kreis};
\node (v8) at (2*\csep,4*\rsep) {\kreis};
\node (v9) at (2*\csep,3*\rsep) {\kreis};
\node (v10) at (2*\csep,2*\rsep) {\kreis};
\node (v11) at (2*\csep,1*\rsep) {\kreis};
\node (v12) at (2*\csep,0*\rsep) {\kreis};

\draw [->] 	(v0) -- node [midway,fill=white,sloped,xshift=3] {$+-$} (v2);
\draw [->] 	(v0) -- node [above, sloped] {$++$} (v1);
\draw [->] 	(v0) -- node [midway,fill=white,sloped,xshift=3] {$-+$} (v3);
\draw [->] 	(v0) -- node [below, sloped] {$--$} (v4);

\draw [->] 	(v1) -- node [xshift=8,fill=white] {$-$} (v6);
\draw [->] 	(v1) -- node [above, sloped] {$+$} (v5);
\draw [->] 	(v2) -- node [xshift=8,fill=white,sloped] {$+$} (v7);
\draw [->] 	(v2) -- node [below,sloped,yshift=2] {$-$} (v8);
\draw [->] 	(v3) -- node [above,sloped,yshift=-2] {$+$} (v9);
\draw [->] 	(v3) -- node [xshift=8,fill=white,sloped] {$-$} (v10);
\draw [->] 	(v4) -- node [xshift=8,fill=white] {$+$} (v11);
\draw [->] 	(v4) -- node [below, sloped] {$-$} (v12);

\node at (0.5*\csep,7.5*\rsep) [text width=3cm] {socio-economic\\ background};
\node at (1.5*\csep,8*\rsep) {life events};
\end{tikzpicture}
\caption{A staged tree $\pt$, simplified version taken from \cite{lorna:bnCEG}. We label the edges by $+$ and $-$, corresponding to \enquote{high} and \enquote{low}, respectively. See \cref{bsp:lorna1} for a discussion.}\label{fig:lorna1}
\end{figure}
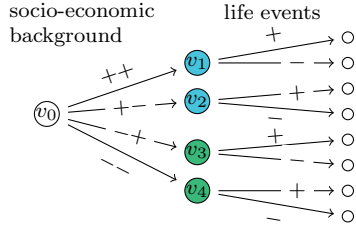

\section{A Polynomial Characterization of Staged Tree Models}\label{sect:polynomials}
In the development below we will interpret primitive probabilities or parameters which determine statistical models not as place holders for as yet undetermined numerical values but as elements of some formal symbolic (or algebraic) structure. This has been a hugely successful approach in the field of algebraic statistics~\citep{oberwolfach,pistonerw:algstat} which has until now not been explored for staged trees.
We have found that the stage constraints on the parameter space of a staged tree can be easily translated into polynomial constraints on the atomic probabilities, and that the underlying model can thus be characterized as the solution set of a set of polynomial equations. This result is analogous to a well-known algebraic characterisation of BN models~\cite{geigermeeksturm:algebra}. The process of translating the model-defining equations in \cref{eq:primprob} into equations which do not depend on a fixed parametrisation is straightforward but technical: see \cite{mythesis}.

So in this paper we will instead use the idea of embedding model assumptions in an algebraic framework to develop an alternative and rather different approach which is more intuitive for staged trees. In particular, we define a polynomial below which can be used to both represent a staged tree model and to recover \emph{constructively} all possible graph representations: a process not possible using only a set of defining equations.

\subsection{Polynomial Equivalence and the Swap Operator}\label{sub:polyequivalence}
We first characterize a subclass of a class of statistically equivalent staged trees for which the following polynomial is invariant:

\begin{mydef}[Interpolating polynomial]\label{def:interpol}
Let $\mathbb{P}=\{\Ptheta~|~\ttheta\in\Theta\}$ denote a parametric statistical model on an underlying discrete space $\Omega$ with the property that every atomic probability $\Ptheta(\omega)$ is a monomial in the parameters $\ttheta$, $\omega\in\Omega$. A \emph{network polynomial} is of the form
\begin{equation}\label{eq:charpol}
c_{g,\mathbb{P}}(\ttheta)=\sum_{\mathclap{\omega\in\Omega}}g(\omega)\Ptheta(\omega)
\end{equation}
where $g$ is a function that determines a real coefficient for each monomial. If $g=1$, then the symbolic sum of all atomic monomials is called an \emph{interpolating polynomial} of the model $\mathbb{P}$. 
\end{mydef}

Network polynomials where $g=\ind_A$ is chosen to be an indicator of an event $A\subseteq\Omega$ have been successfully used to answer probabilistic queries in BN models~\citep{darwiche:networkpol}. Different choices of $g$ also relate the network polynomial to moment generating functions~\citep{pistone:interpolating}. We have already demonstrated the efficacy of using these polynomials to calculate marginal and conditional probabilities in staged tree models~\cite{manuich}, for sensitivity analysis in models with a multilinear parametrization~\cite{manuich2} and for causal manipulations~\cite{PGM}.

In the following, we will write $c_{\T}(\ttheta)=\sum_{\lambda\in\Lambda(\T)}\pitheta(\lambda)$ for the interpolating polynomial $c_{1,\treemodel}$ of a tree model represented by $\pt$.

\begin{rem}\label{rem:interpol}
When evaluating the network polynomial of a staged tree for a certain choice of parameters, we find that in a non-symbolic framework the polynomial
\begin{equation}\label{eq:prob}
c_{\ind_A,\T}(\ttheta) =\summ{\lambda\in\Lambda(\T)}\ind_A(\lambda)\pitheta(\lambda) 
=\summ{\lambda\in A}\pitheta(\lambda)=\Ptheta(\iotat^{-1}(A))
\end{equation}
is a function $(A,\ttheta)\mapsto \Ptheta(\iotat^{-1}(A))$ which maps an event $A\subseteq\Lambda(\T)$ and a choice of parameters to the probability of that event.
\end{rem}

In a symbolic setting, we usually ignore sum-to-$1$ conditions and exploit only the formal structure of a polynomial. This is for instance beneficial when obtaining results like \cref{eq:prob} from differentiation operations~\cite{manuich}. Note that because any choice of floret sum-to-$1$ conditions on an event tree will yield a probability distribution over the depicted root-to-leaf paths, these conditions can be ignored in the characterisation of equivalence classes of staged trees below and be imposed only after having found a model representation.

\begin{mydef}[Polynomial equivalence]\label{def:polyequiv}
Let $\pt$ and $\st$ be two staged trees with the same underlying space $\Omega$. These staged trees are called \emph{polynomially equivalent} if and only if they have the same edge labels and their network polynomials coincide formally $c_{g,\St}=c_{g,\T}$ for every choice of the function $g$. 
\end{mydef}

In general, polynomial equivalence is not necessary for statistical equivalence: see for instance \cref{bsp:saturated} where two very different parametrizations can be used for the same model. However, we do have the following result:

\begin{lem}\label{lem:polyimpliesstat}
Polynomial equivalence implies statistical equivalence.
\end{lem}

\begin{proof}
Set $g_{\omega}=\ind_{\{\omega\}}$ to be the indicator function of an arbitrary atomic event $\omega\in\Omega$. Then polynomial equivalence of two staged trees $\pt$ and $\st$ implies termwise equality of the probability mass functions $c_{g_{\omega},\T}(\ttheta)=\Ptheta(\omega)=c_{g_{\omega},\St}(\ttheta)$ by \cref{rem:interpol}.
\end{proof}

So all polynomially equivalent staged trees can be characterized as having the same interpolating polynomial plus an identification between their atoms. In square-free staged trees, the probability mass function $\pitheta:\lambda\mapsto\prod_{e\in E(\lambda)}\theta(e)$ is formally injective: that is, the atomic monomials are pairwise different and we can uniquely identify atoms with monomials. So in these trees, the interpolating polynomial is sufficient to identify a class of polynomially equivalent staged trees.
We henceforth use the symbol $\ptrep^c\subseteq\ptrep$ to denote a class of polynomially equivalent square-free staged trees which share the same interpolating polynomial $c$. Note that a staged tree $\pt$ is square-free if and only if $c_{\T}$ is linear in every indeterminate. We will restrict all further analysis to this class of models.

\begin{rem}\label{rem:nesting}
The graph of a staged tree $\pt$ yields a way to parenthesize the associated interpolating polynomial as follows.
For every floret $\F_v$ where $v\in V$ is the parent of a leaf, we sum all components of its parameter vector $\ttheta_v$ and multiply the result by its parent label $\theta(\pa(v),v)$. We then sum the result over the parent's labels $\ttheta_{\pa(v)}$. By repeating this until all floret parameter vectors are summed and $\pa(v)=v_0$, the interpolating polynomial can then be written in terms of a nested factorization
\begin{equation}\label{eq:interpolfactored}
c_\T(\ttheta)~=~\sum_{\mathclap{v_1\in\ch(v_0)}}\theta(v_0,v_1) \Bigl(\sum_{\mathclap{v_2\in\ch(v_1)}}\theta(v_1,v_2)~\ldots~ \Bigl(\sum_{\mathclap{v_k\in\ch(v_{k-1})}}\theta(v_{k-1},v_k) \Bigr)\Bigr)
\end{equation}
where the final index $k\in\N$ of every inner sum implicitly depends on the length of a root-to-leaf path $((v_0,v_1),(v_1,v_2),\ldots,(v_{k-1},v_k))$. See \cref{fig:polyequiv} for an illustration.
\end{rem}

The interpolating polynomials of discrete BN models admit a nested bracketing as in \cref{eq:interpolfactored}. The parameters in those polynomials are then \emph{potentials} of a probability mass function  and are normalized via the florets of an underlying tree representation. This type of representation of a polynomial provides a very efficient way to compute joint probabilities from marginals in a BN model~\cite{lauritzenspiegelhalter} and comes for free when choosing a staged tree representation rather than an acyclic digraph.

Centrally, we can generalize the observation above to a new concept:

\begin{mydef}[Tree compatibility]\label{def:treecompatible}
Let $\ttheta=(\theta_1,\ldots,\theta_d)$ be a parameter vector, $d\in\N$. We call any polynomial \emph{tree compatible} if it admits a representation of the form
\begin{equation}\label{eq:treecompatible}
c(\ttheta)~=~\summ{\theta_1\in A_1~}\theta_1\Bigl(\sum_{\mathclap{~\theta_2\in A_2(\theta_1)~}}\theta_2\Bigl(\sum_{\mathclap{\quad~~\theta_3\in A_3(\theta_2)}}\theta_3~\ldots~\Bigl(\sum_{\mathclap{~\theta_k\in A_k(\theta_{k-1})}}\theta_{k}\Bigr)\Bigr)\Bigr)
\end{equation}
where every $A_1,A_j(\theta_{j-1})\subseteq\{\theta_1,\ldots,\theta_d\}$ has at least two elements, for $j\in[k]$ and $k\in\N$. We write $\fact$ for one fixed order of summation of the terms in $c(\ttheta)$ as above, and call this a \emph{tree-compatible factorization}.
\end{mydef}

An important aspect of the result in \cref{rem:nesting} is that it is reversible: not only can we easily read a polynomial from an event tree but we can also construct a tree graph from a tree-compatible factorization. In addition, all polynomially equivalent staged trees arise from a tree-compatible reordering of a given summation. Each of these gives a different representation within the same statistical equivalence class.

\begin{prop}\label{prop:polyequiv}
Let $\mathbb{P}$ be a discrete parametric model whose atomic probabilities are of monomial form and let $c=c_{1,\mathbb{P}}$ denote its interpolating polynomial. Then there exists a probability tree representation $\pt$ with $\treemodel=\mathbb{P}$ if and only if $c$ is tree compatible. 
The map $\emap:\fact\mapsto\pt$ is invertible.
\end{prop}

\begin{proof}
Sufficiency of the first part of the claim is straightforward because the interpolating polynomial of a tree model is tree compatible by \cref{eq:interpolfactored}.

For necessity assume now the interpolating polynomial of a parametric model to be tree compatible and given by the factorization $\fact$ in \cref{eq:treecompatible}.
We construct a labelled graph as follows: for every subsum of \cref{eq:treecompatible}, draw a floret $\F_j=(v_j,\{e~|~\theta(e)=\theta_j\in A_j(\theta_{j-1})\})$ with one edge for every indeterminate in the sum and attach these indeterminates as labels, $j\in[k]$. Then partially order these florets by reversing the steps in \cref{rem:nesting}, such that $\theta_j$ is the parent label of the floret whose attached parameters are $A(\theta_{j})$, for all $j\in[k]$. In this way, we construct a connected graph with no cycles---and hence a tree---whose leaf-floret edges are labelled by the innermost factors $A_k(\theta_{k-1})$ of $\fact$ and the root's edges by the outermost factors $A_1$. Since by definition every set $A_j(\theta_{j-1})$ has at least two elements, it follows that there are at least two edges in every floret. So the tree-compatible factorisations in \cref{eq:interpolfactored,eq:treecompatible} are componentwise equal and, multiplying out the brackets of $\fact$, we find that $c=c_{\T}$. Thus, we have constructed a labelled event tree $\pt$. Then imposing sum-to-$1$ conditions on the constructed florets as in \cref{def:probtree} is consistent with the sum-to-$1$ conditions on atomic probabilities in the model. So $\pt$ is a probability tree which represents $\mathbb{P}=\treemodel$.

By construction, the steps above are reversible. So the map which identifies a tree-compatible factorisation with a labelled tree is invertible.
\end{proof}

The proposition provides us with a powerful tool to determine whether a parametric model can be represented by a probability tree. This representation is a staged tree only if all constraints on the model are of the form $A_{i+1}(\theta_i)=A_{j+1}(\theta_j)$ for some $i\not=j$ in the notation of \cref{def:treecompatible}.

The result above induces two natural streams of research. First, how can we check whether or not a given interpolating polynomial is tree compatible? We will discuss this issue at the very end of this work where we will also outline ideas for an algorithmic implementation.

The second question is: how do we infer all the possible orders of bracketing of a tree-compatible interpolating polynomial $c_{\T}$? Knowing this, we can use the map $\emap$ in \cref{prop:polyequiv} and the construction outlined in the proof to obtain all tree representations in $\ptrep^c$. We will show how to do this below.
\medskip

Clearly, a transformation between two tree-compatible factorizations of an interpolating polynomial is an application of the distributive property of addition and multiplication. These correspond to the following intuitive graph transformation.

We henceforth call a subgraph of a probability tree which is an event tree with inherited edge labels a \emph{(probability) subtree}.
We call a probability subtree $\pt_u\subseteq\pt$ a \emph{twin} if it is of the following form: all root-to-leaf paths consist of exactly two edges and all children of its root are in the same stage $u$. This stage does not contain the root itself.

The interpolating polynomial of a twin $\pt_u$ can be written in the form
\begin{equation}\label{eq:swap}
c_{\T_u}(\ttheta)=\summ{e\in E(v_0)}\theta(e)\Bigl(\summ{e'\in E(v)}\theta(e')\Bigr)=\summ{e'\in E(v)}\theta(e')\Bigl(\summ{e\in E(v_0)}\theta(e)\Bigr)
\end{equation}
where $v\in u$ is one representative of the stage $u=\ch(v_0)$ and $v_0$ is the root of the twin. By \cref{prop:polyequiv}, there is a staged tree $\st_u$ which is polynomially equivalent to $\pt_u$: this is the one given by the second tree-compatible factorization in \cref{eq:swap}. Then, $\st_u$ is a subtree of a tree $\st$ which is polynomially equivalent to $\pt$ and coincides with that tree everywhere except on $\pt_u$.

\begin{mydef}[Swap]\label{def:swap}
Let $\pt$ be a staged tree and $\pt_u\subseteq\pt$ a twin around the stage $u$. Denote by $\st_u$ the staged tree which is polynomially equivalent to $\pt_u$ and let $\st_u\subseteq\st$ as above.
We will call the map ${\smap:\pt\mapsto\st}$ a \emph{na\"ive swap} and call it a \emph{swap} if $\st$ is itself a staged tree.
\end{mydef}

\Cref{fig:polyequiv} illustrates the definition above. We can see here that this operation does indeed \enquote{swap} the order of edges before and after the stage $u$. It is straightforward to show that edge-centred events on the root-edges of a twin are independent of those of edges ending in leaves. Our very plausible discovery is that for these independent events the order in which they are depicted in a tree is reversible within a statistical equivalence class, using the swap operator. 

Whilst it is natural for swaps to change the floret structure, as explained below, na\"ive swaps might also violate stage structure. The simplest case is when the root of a twin is in a stage in the original tree, and a na\"ive swap rearranges that floret but not an identified one elsewhere in the graph.
\begin{figure}[tb]
\centering
\subfloat[A staged tree $\pt$ with twin $\pt_u$ thick depicted, $u=\{v_1,v_2\}$.]{
\centering
\begin{tikzpicture}
\renewcommand{\rsep}{0.23}
\renewcommand{\csep}{1.8}

\node (v0) at (0,4*\rsep) {\vertex{0}};
\node (v1) at (\csep,4*\rsep) {\stage{SkyBlue}{1}};
\node (v2) at (\csep,2*\rsep) {\stage{SkyBlue}{2}};
\node (v3) at (2*\csep,2*\rsep) {\vertex{3}};

\node (l1) at (\csep,6*\rsep) {\kreis};
\node (l2) at (2*\csep,6*\rsep) {\kreis};
\node (l3) at (2*\csep,4*\rsep) {\kreis};
\node (l4) at (2*\csep,0) {\kreis};
\node (l5) at (3*\csep,4*\rsep) {\kreis};
\node (l6) at (3*\csep,2*\rsep) {\kreis};
\node (l7) at (3*\csep,0*\rsep) {\kreis};

\draw [->, ultra thick]	(v0) -- node [xshift=4,fill=white] {$\theta_2$} (v1);
\draw [->] 	(v0) -- node [above, sloped] {$\theta_1$} (l1);
\draw [->, ultra thick]	(v0) -- node [below, sloped] {$\theta_3$} (v2);
\draw [->, ultra thick]	(v1) -- node [xshift=4,fill=white] {$\theta_5$} (l3);
\draw [->, ultra thick]	(v1) -- node [above, sloped] {$\theta_4$} (l2);
\draw [->, ultra thick]	(v2) -- node [xshift=4,fill=white] {$\theta_4$} (v3);
\draw [->, ultra thick]	(v2) -- node [below,sloped] {$\theta_5$} (l4);
\draw [->]	(v3) -- node [above, sloped] {$\theta_6$} (l5);
\draw [->]	(v3) -- node [xshift=5,fill=white] {$\theta_7$} (l6);
\draw [->]	(v3) -- node [below, sloped] {$\theta_8$} (l7);
\end{tikzpicture}}
\qquad
\subfloat[A staged tree $\st$ with twin $\st_u$ thick depicted, $u=\{v_1,v_2\}$.]{
\centering
\begin{tikzpicture}
\renewcommand{\rsep}{0.23}
\renewcommand{\csep}{1.8}

\node (v0) at (0,4*\rsep) {\vertex{0}};
\node (v1) at (\csep,4*\rsep) {\stage{Green!70}{1}};
\node (v2) at (\csep,2*\rsep) {\stage{Green!70}{2}};
\node (v3) at (2*\csep,4*\rsep) {\vertex{3}};

\node (l1) at (\csep,6*\rsep) {\kreis};
\node (l2) at (2*\csep,6*\rsep) {\kreis};
\node (l3) at (2*\csep,2*\rsep) {\kreis};
\node (l4) at (2*\csep,0) {\kreis};
\node (l5) at (3*\csep,6*\rsep) {\kreis};
\node (l6) at (3*\csep,4*\rsep) {\kreis};
\node (l7) at (3*\csep,2*\rsep) {\kreis};

\draw [->, ultra thick]	(v0) -- node [xshift=4,yshift=-1,fill=white] {$\theta_4'$}(v1);
\draw [->] 	(v0) -- node [above, sloped] {$\theta_1'$} (l1); 
\draw [->, ultra thick]	(v0) -- node [below, sloped] {$\theta_5'$} (v2);
\draw [->, ultra thick]	(v1) -- node [xshift=4,yshift=-1,fill=white] {$\theta_3'$} (v3);
\draw [->, ultra thick]	(v1) -- node [above, sloped] {$\theta_2'$} (l2);
\draw [->, ultra thick]	(v2) -- node [xshift=4,yshift=-1,fill=white] {$\theta_2'$} (l3);
\draw [->, ultra thick]	(v2) -- node [below, sloped] {$\theta_3'$} (l4);
\draw [->]	(v3) -- node [xshift=5,fill=white] {$\theta_7'$} (l6);
\draw [->]	(v3) -- node [above, sloped] {$\theta_6'$} (l5);
\draw [->]	(v3) -- node [below, sloped] {$\theta_8'$} (l7);
\end{tikzpicture}}
\caption{Two polynomially equivalent staged trees with the same indeterminates $\theta_i=\theta_i'$, $i=1,\ldots,8$, and interpolating polynomials $c_{\T}(\ttheta)=\theta_1+(\theta_2(\theta_4+\theta_5)+ \theta_3(\theta_4(\theta_6+\theta_7+\theta_8)+\theta_5))$ and $c_{\St}(\ttheta')=\theta_1'+\theta_4'(\theta_2'+\theta_3' (\theta_6'+\theta_7'+\theta_8'))+\theta_5'(\theta_2'+\theta_3'))$ given in the respective tree-compatible factorization. The map $\smap:\pt\mapsto\st$ is a swap.}\label{fig:polyequiv}
\end{figure}
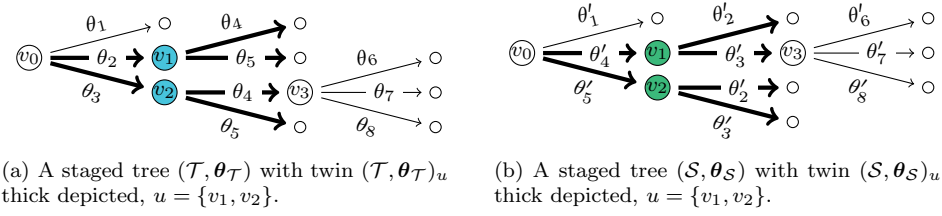
We henceforth call a composition of swaps for which floret parameter vectors are invariant a \emph{floret-swap} and a composition of swaps which swaps all edges at a fixed distance from the root a \emph{level-swap}. For instance, the swap in \cref{fig:polyequiv}\label{exp:swap} is not a floret-swap because the root-vector $(\theta_1,\theta_2,\theta_3)$ is not a vector in $\st$.  Conversely, $(\theta_1,\theta_4,\theta_5)$ is not a floret parameter vector in $\pt$. This implies that $\pt$ and $\st$ have different local sum-to-$1$ conditions on their primitive probabilities. By \cref{lem:polyimpliesstat}, both represent the same model. So even if the numerical value of say $\theta_1=\theta(e_1)$ is different in $\pt$ and $\st$---we have indicated this using labels $\theta_1=\theta_1'$ in \cref{fig:polyequiv}---via a renormalization it is still the probability of the event $\iotat^{-1}(\{\lambda\in\Lambda(\T)~|~e_1\in E(\lambda)\})\subseteq\Omega$ depicted by both graphs. The meaning of this parameter is thus unchanged and can be identified across different representations.
\medskip

We can now obtain the following result, which enables us to both graphically and algebraically move around a class of polynomially equivalent trees.

\begin{prop}\label{prop:swaps}
Two square-free staged trees $\pt$ and $\st$ are polynomially equivalent if and only if there exists a finite composition of na\"ive swaps $\smap_1,\ldots,\smap_l$, $l\in\N$, for which $\smap_l\circ\smap_{l-1}\circ\ldots\circ\smap_1:\pt\mapsto\st$ is a swap.
\end{prop}

\begin{proof}
Let $\pt=\emap(s_1(c(\ttheta)))$ and $\st=\emap(s_2(c(\ttheta)))$ be polynomially equivalent staged trees with a common interpolating polynomial $c$ and corresponding tree-compatible factorizations $s_1$ and $s_2$ as in \cref{eq:interpolfactored}. Here, $\emap$ denotes the map from \cref{prop:polyequiv}.
Clearly, one factorization $s_1(c(\ttheta))$ is transformed into the other $s_2(c(\ttheta))$ by applying the distributive law of $+$ and $\cdot$ a finite number of times. Hence, we can define a map $\tilde{s}:s_1(c(\ttheta))\mapsto s_2(c(\ttheta))$ performing these calculations on the subsums of $c$ as in~\cref{eq:swap}. Therefore, 
\begin{equation}
\smap:~\pt\overset{\emap^{-1}}{\mapsto}s_1(c(\ttheta))\overset{\tilde{s}}{\mapsto}s_2(c(\ttheta))\overset{\emap^{-1}}{\mapsto}\st
\end{equation}
is a map which performs a finite number of swaps on the to $\tilde{s}$ corresponding twins and thus transforms $\pt$ into $\st$.
\end{proof}

Thus, the polynomial equivalence class of a staged tree can be fully traversed by an algebraic resummation operation or, equivalently, by local graph transformations. Note that this operator is a close analogue to the \emph{arc reversal} in BN models~\cite{schachter}. These, just like swaps, allow us to traverse the class of all graphical representations of the same model, while renormalizing (but not marginalizing) the associated probability mass function.

\begin{bsp}[\Cref{bsp:lorna1} continued]\label{bsp:twistsandarcs}
The staged tree $\pt$ in \cref{fig:lorna1} contains two twins: $\pt_{\text{blue}}$ around the stage $u_{\text{blue}}=\{v_1,v_2\}$ and $\pt_{\text{green}}$ around the stage $u_{\text{green}}=\{v_3,v_4\}$.
Applying a level-swap on $\pt$ which swaps both of these twins, we obtain a new tree $\st_1$ depicted in \cref{fig:arcreversal}. In $\st_1$, the edges emanating from the root now correspond to the random variable \enquote{social background and life events} rather than \enquote{social and economic background}. This application of a swap corresponds to a renormalization of an underlying probability mass function $p(s,e,l)=p(s,e)p(l|s)$ to $p(s,l)p(e|s)$, for $s,e,l\in\{\text{high},\text{low}\}$. This result can equivalently be achieved by applying an arc reversal on an alternative representation of this model in terms of a decomposable acyclic digraph.

Unlike $\st_1$, the staged tree $\st_2$ in \cref{fig:twist} where a swap has been applied only on $\pt_{\text{blue}}$, cannot be straightforwardly identified with a BN model. In particular, the edges emanating from the root now correspond to a new random variable $X$ \enquote{life events in low social background and economic situation in high social background}, and the variable $Y$ associated to leaf-edges changes accordingly:
\begin{equation*}
X=\begin{cases}(S,L) &\text{if }S=0\\ (S,E) &\text{if }S=1\end{cases}\quad\text{and }~ Y=\begin{cases}E|L &\text{if }S=0\\ L|E&\text{if }S=1.\end{cases}
\end{equation*}
\end{bsp}

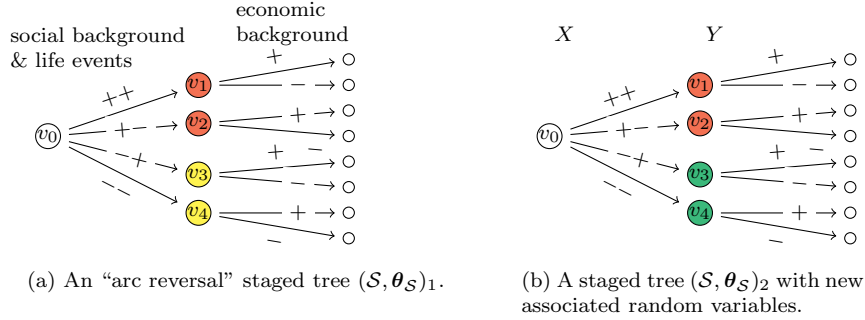
\begin{figure}[tb]
\centering
\subfloat[An \enquote{arc reversal} staged tree $\st_1$.]{
\centering
\begin{tikzpicture}
\renewcommand{\csep}{2}
\renewcommand{\rsep}{0.34}
\node (v0) at (0*\csep,4*\rsep) {\vertex{0}};
\node (v1) at (1*\csep,6*\rsep) {\stage{Red!70}{1}};
\node (v2) at (1*\csep,4.5*\rsep) {\stage{Red!70}{2}};
\node (v3) at (1*\csep,2.5*\rsep) {\stage{Yellow!80}{3}};
\node (v4) at (1*\csep,1*\rsep) {\stage{Yellow!80}{4}};
\node (v5) at (2*\csep,7*\rsep) {\kreis};
\node (v6) at (2*\csep,6*\rsep) {\kreis};
\node (v7) at (2*\csep,5*\rsep) {\kreis};
\node (v8) at (2*\csep,4*\rsep) {\kreis};
\node (v9) at (2*\csep,3*\rsep) {\kreis};
\node (v10) at (2*\csep,2*\rsep) {\kreis};
\node (v11) at (2*\csep,1*\rsep) {\kreis};
\node (v12) at (2*\csep,0*\rsep) {\kreis};

\draw [->] 	(v0) -- node [above, sloped] {$++$} (v1);
\draw [->] 	(v0) -- node [xshift=3,fill=white,sloped] {$+-$} (v2);
\draw [->] 	(v0) -- node [xshift=3,fill=white,sloped] {$-+$} (v3);
\draw [->] 	(v0) -- node [below, sloped] {$--$} (v4);

\draw [->] 	(v1) -- node [xshift=8,fill=white] {$-$} (v6);
\draw [->] 	(v1) -- node [above, sloped] {$+$} (v5);
\draw [->] 	(v2) -- node [xshift=8,fill=white,sloped] {$+$} (v7);
\draw [->] 	(v2) -- node [below, sloped, xshift=15] {$-$} (v8);
\draw [->] 	(v3) -- node [above, sloped] {$+$} (v9);
\draw [->] 	(v3) -- node [xshift=8,fill=white,sloped] {$-$} (v10);
\draw [->] 	(v4) -- node [xshift=8,fill=white] {$+$} (v11);
\draw [->] 	(v4) -- node [below, sloped] {$-$} (v12);

\node at (0.5*\csep,7.5*\rsep) [text width=3cm] {social background\\ \& life events};
\node at (2*\csep,8.4*\rsep) [text width=3cm] {economic\\ background};
\end{tikzpicture}
\label{fig:arcreversal}}
\qquad
\subfloat[A staged tree $\st_2$ with new associated random variables.]{
\centering
\begin{tikzpicture}
\renewcommand{\csep}{2}
\renewcommand{\rsep}{0.34}
\node (v0) at (0*\csep,4*\rsep) {\vertex{0}};
\node (v1) at (1*\csep,6*\rsep) {\stage{Red!70}{1}};
\node (v2) at (1*\csep,4.5*\rsep) {\stage{Red!70}{2}};
\node (v3) at (1*\csep,2.5*\rsep) {\stage{Green!70}{3}};
\node (v4) at (1*\csep,1*\rsep) {\stage{Green!70}{4}};
\node (v5) at (2*\csep,7*\rsep) {\kreis};
\node (v6) at (2*\csep,6*\rsep) {\kreis};
\node (v7) at (2*\csep,5*\rsep) {\kreis};
\node (v8) at (2*\csep,4*\rsep) {\kreis};
\node (v9) at (2*\csep,3*\rsep) {\kreis};
\node (v10) at (2*\csep,2*\rsep) {\kreis};
\node (v11) at (2*\csep,1*\rsep) {\kreis};
\node (v12) at (2*\csep,0*\rsep) {\kreis};

\draw [->] 	(v0) -- node [above, sloped] {$++$} (v1);
\draw [->] 	(v0) -- node [xshift=3,fill=white,sloped] {$+-$} (v2);
\draw [->] 	(v0) -- node [xshift=3,fill=white,sloped] {$-+$} (v3);
\draw [->] 	(v0) -- node [below, sloped] {$--$} (v4);

\draw [->] 	(v1) -- node [xshift=8,fill=white] {$-$} (v6);
\draw [->] 	(v1) -- node [above, sloped] {$+$} (v5);
\draw [->] 	(v2) -- node [xshift=8,fill=white,sloped] {$+$} (v7);
\draw [->] 	(v2) -- node [below, sloped, xshift=15] {$-$} (v8);
\draw [->] 	(v3) -- node [above, sloped] {$+$} (v9);
\draw [->] 	(v3) -- node [xshift=8,fill=white,sloped] {$-$} (v10);
\draw [->] 	(v4) -- node [xshift=8,fill=white] {$+$} (v11);
\draw [->] 	(v4) -- node [below, sloped] {$-$} (v12);

\node at (0.1*\csep,8*\rsep) {$X$};
\node at (1.1*\csep,8*\rsep) {$Y$};
\end{tikzpicture}\label{fig:twist}}
\caption{Two staged trees which are polynomially equivalent to the one from \cref{bsp:lorna1}. See \cref{bsp:twistsandarcs}.}
\end{figure}

The example above provides a very simple illustration of how the statistical equivalence class of a staged tree (or CEG) can be so much larger than that of a BN. It also demonstrates how staged trees can implicitly generate relationships between new random variables, constructed as functions of the original ones: possibly useful in later interpretative analysis. A more detailed discussion of this process is given in \cref{sect:causal}. 

\subsection{Statistical Equivalence and the Resize Operator}\label{sub:statequiv}

We have seen in \cref{rem:interpol} that the network polynomial can be used to calculate probabilities of events represented by a staged tree. Thus, when leaving the symbolic framework and substituting values for the edge parameters, the interpolating polynomial is clearly invariant for a class of statistically equivalent staged trees. So in order to extend our characterization of polynomial equivalence classes to the whole statistical equivalence class, we will need to reparametrize between two given descriptions without violating the model assumptions. We define a second operator below which will again enable us to achieve this aim constructively.\medskip

We henceforth call the pair $\pt'\subseteq\pt$ a \emph{subgraph} if it is a subtree with inherited edge labels whose root might have only one emanating edge, not necessarily two as required in an event tree.

\begin{mydef}[Resize]\label{def:resize}
Let $\pt$ be a staged tree and let $\pt'\subseteq\pt$ be a subgraph. We denote by $\fmap:~\pt\mapsto\st$ the map which transforms this subgraph into a floret $\ft$ with labels $\ttheta_{\F}=\bigl(\pi_{\ttheta,\T'}(\lambda')~|~\lambda'\in\Lambda(\T')\bigr)$, while leaving the remaining graph invariant.
We call $\fmap$ and its inverse $\fmap^{-1}$ \emph{na\"ive resize} operators, and a \emph{resize} if $\st$ is a staged tree.
\end{mydef}

In terms of the atomic monomials, a na\"ive resize performs a substitution of products of primitive probabilities into degree $1$ monomials. By construction, atomic probabilities of root-to-leaf paths are invariant under this operation: see also \cref{bsp:saturated}. However, $\fmap$ is not necessarily a well-defined map between two \emph{staged} trees. The lemma below establishes useful criteria for a well-defined application of this operator.

\begin{lem}\label{lem:naiveresize}
Let $\pt$ be a staged tree and $l\in\N$. A composition of na\"ive resizes $\fmap=\fmap_l\circ\ldots\circ\fmap_1$ applied to $\pt$ is a resize if one of the following conditions is fulfilled:
\begin{enumerate}[a)]
\item $\fmap$ only acts on saturated subgraphs.
\item $\fmap$ only acts on subgraphs which are polynomially equivalent to each other and whose vertices are not in the same stage as vertices outside these subgraphs.
\end{enumerate}
\end{lem}

\begin{proof}
a) Because the image $\fmap\pt=\st$ of a staged tree is a probability tree and because by assumption the stage sets of image and preimage coincide, clearly also $\st\in\ptrep$ is a staged tree.

b) Because all subgraphs $\pt'$, $\pt''\subseteq\pt$ that $\fmap$ acts on are polynomially equivalent, they are also statistically equivalent: see \cref{lem:polyimpliesstat}. So after resizing, we identify the atomic probabilities $\pi_{\ttheta,\T'}(\lambda')=\pi_{\ttheta,\T''}(\lambda'')$ of subpaths $\lambda'$, $\lambda''$ which have the same atomic monomial in $\pt$. Thus, the image $\st=\fmap\pt$ is a staged tree where the stages are given by these identified (formerly atomic now) primitive probabilities.
\end{proof}

Note that case (a) in \cref{lem:naiveresize} enables us to contract subgraphs which do not contain any stage information, so are in that sense not informative to the model. Analogous operations are often performed on BN models where the cliques of a decomposable model do not contain any conditional independence information and can hence be treated as a joint random variable without leaving the model class~\cite{lauritzen}. Case (b) enables us to directly identify atomic monomials of polynomially equivalent subgraphs rather than repeating stage equations edge by edge. 
Note that if these conditions are violated, then a na\"ive resize can take us out of the statistical equivalence class of a staged tree.

\begin{lem}\label{lem:resize}
Let $\pt$ be a staged tree and $\fmap$ a resize operator, possibly a composition of na\"ive resizes. Then $\pt$ and $\fmap\pt$ are statistically equivalent staged trees.
\end{lem}

This results follows immediately from the definition.

Finally, the resize in conjunction with the swap operator now enables us to traverse the whole equivalence class of a given staged tree.

\begin{thm}\label{thm:thisisit}
Two square-free staged trees $\pt$ and $\st$ are statistically equivalent if and only if there exists a map $\mathfrak{m}:\pt\mapsto\st$ which is a finite composition of resizes and swaps.
\end{thm}

\begin{proof}
First let $\pt$ and $\st$ be statistically equivalent staged trees. Then all identified root-to-leaf paths $\lambda'=\iota_{\St}(\iotat(\lambda))$ have equal atomic probabilities, $\pi_{\ttheta,\T}(\lambda)=\pi_{\ttheta',\St}(\lambda')$. If the above equality holds in a formal sense for every $\lambda\in\Lambda(\T)$ then $\pt$ and $\st$ are polynomially equivalent. In this case, \cref{lem:polyimpliesstat} states that a map exists between the two staged trees which is a composition of swaps, and thus proves the claim. If this is not the case, we denote by $\Lambda\subseteq\Lambda(\T)$ the set of root-to-leaf paths in $\T$ whose atomic monomials do not coincide formally with the corresponding atomic monomials in $\St$. Let $\pt'\subseteq\pt$ denote a subtree of $\T$ for which $\Lambda\subseteq\Lambda(\T')$, and define analogously the corresponding $\st'\subseteq\st$. These are the subtrees which are not polynomially equivalent. We define two na\"ive resize operators, $\fmap_{\T}:\pt'\mapsto\ft$ and $\fmap_{\St}:\st'\mapsto\ft$ which map those subtrees to the same floret. By \cref{lem:resize}, $\st'$, $\pt'$ and $\ft$ are statistically equivalent. Thus, there is a composition of resizes $\fmap=\fmap_{\St}^{-1}\circ\fmap_{\T}:\pt\mapsto\st$ between the statistically equivalent staged trees.

Now let $\mathfrak{m}$ be a transformation given by swaps and resizes between two staged trees $\pt$ and $\st$. If $\mathfrak{m}$ is a composition of swaps, then \cref{prop:polyequiv} ensures polynomial equivalence, and thus statistical equivalence by \cref{lem:polyimpliesstat}. If $\mathfrak{m}$ is a composition of resizes, then \cref{lem:resize} yields statistical equivalence. Clearly, also for the composition of both of these operators holds that $\pt$ and $\mathfrak{m}\pt=\st$ are statistically equivalent. The claim follows.
\end{proof}

\section{Analyzing a full Statistical Equivalence Class}\label{sect:causal}
We will now characterize properties of the statistical equivalence class of a staged tree inferred from a dataset. In particular, using the resize operator we will create new random variables describing the system and using the swap operator we will be able to give a putative causal interpretation to a depicted order of events.
\medskip

A staged tree model for the Christchurch Health and Development Study (CHDS)~\cite{chds} has been closely analyzed, e.g.\,in~\cite{lorna:bnCEG,cowelljim:causal}, and has been used to describe the interplay of the social support, the economic situation, hospital admissions and possible life events (e.g.\,divorce, redundancy of a parent) of a group of children over a fixed period of time. The staged tree $\pt$ in \cref{fig:majormap} was found using an MAP search~\cite{cowelljim:causal}. We will now apply \cref{thm:thisisit} to the statistical equivalence class $\ptrep$ in order to enrich our understanding of the model $\treemodel$ represented by that tree.

\begin{figure}[t]
\centering
\subfloat[The MAP staged tree $\pt$ from \cite{cowelljim:causal}.]{
\centering
\begin{tikzpicture}
\renewcommand{\csep}{1.7}
\renewcommand{\rsep}{0.23}

\node at (0.5*\csep,24.6*\rsep) [text width=2*\csep] {\mbox{background:}};
\node at (0.5*\csep,22.6*\rsep) [text width=1*\csep] {social};
\node at (1.2*\csep,22.6*\rsep) [text width=1*\csep] {economic};
\node at (2.2*\csep,25.6*\rsep) [text width=1*\csep] {hospital \mbox{admission}};
\node at (3.6*\csep,25.9*\rsep) {life events};

\node (v0) at (0.5*\csep,11*\rsep) {\vertex{0}};
\node (i1) at (1.25*\csep,16.5*\rsep) {\vertex{1}};
\node (i2) at (1.25*\csep,6.5*\rsep) {\vertex{2}};
\draw [dashed, ->] 	(v0) -- node [above, sloped] {\ls{high}} (i1);
\draw [dashed, ->] 	(v0) -- node [below, sloped] {\ls{low}} (i2);

\node (v1) at (2*\csep,20.5*\rsep) {\stage{Red!70}{3}};
\node (v2) at (2*\csep,14.5*\rsep) {\stage{Red!70}{4}};
\node (v3) at (2*\csep,8.5*\rsep) {\stage{Red!70}{5}};
\node (v4) at (2*\csep,2.5*\rsep) {\vertex{6}};
\draw [dashed, ->] 	(i1) -- node [above, sloped] {\ls{high}} (v1);
\draw [dashed, ->] 	(i1) -- node [below, sloped] {\ls{low}} (v2);
\draw [dashed, ->] 	(i2) -- node [above, sloped] {\ls{high}} (v3);
\draw [dashed, ->] 	(i2) -- node [below, sloped] {\ls{low}} (v4);

\node (w1) at (3*\csep,22*\rsep) {\stage{SkyBlue}{7}};
\node (w2) at (3*\csep,19*\rsep) {\stage{SkyBlue}{8}};
\node (w3) at (3*\csep,16*\rsep) {\stage{SkyBlue}{9}};
\node (w4) at (3*\csep,13*\rsep) {\stage{Green!70}{\text{{\fontsize{1mm}{2mm}\selectfont 10}}}};
\node (w5) at (3*\csep,10*\rsep) {\stage{SkyBlue}{\text{{\fontsize{1mm}{2mm}\selectfont 11}}}};
\node (w6) at (3*\csep,7*\rsep) {\stage{Green!70}{\text{{\fontsize{1mm}{2mm}\selectfont 12}}}};
\node (w7) at (3*\csep,4*\rsep) {\stage{Green!70}{\text{{\fontsize{1mm}{2mm}\selectfont 13}}}};
\node (w8) at (3*\csep,1*\rsep) {\stage{Green!70}{\text{{\fontsize{1mm}{2mm}\selectfont 14}}}};
\draw [->] 	(v1) -- node [above, sloped] {\ls{yes}} (w1);
\draw [->] 	(v1) -- node [below, sloped] {\ls{no}} (w2);
\draw [->] 	(v2) -- node [above, sloped] {\ls{yes}} (w3);
\draw [->] 	(v2) -- node [below, sloped] {\ls{no}} (w4);
\draw [->] 	(v3) -- node [above, sloped] {\ls{yes}} (w5);
\draw [->] 	(v3) -- node [below, sloped] {\ls{no}} (w6);
\draw [dashed, ->] 	(v4) -- node [above, sloped] {\ls{yes}} (w7);
\draw [dashed, ->] 	(v4) -- node [below, sloped] {\ls{no}} (w8);

\node (l1) at (4*\csep,23*\rsep) {\kreis};
\node (l2) at (4*\csep,22*\rsep) {\kreis};
\node (l3) at (4*\csep,21*\rsep) {\kreis};
\node (l4) at (4*\csep,20*\rsep) {\kreis};
\node (l5) at (4*\csep,19*\rsep) {\kreis};
\node (l6) at (4*\csep,18*\rsep) {\kreis};
\node (l7) at (4*\csep,17*\rsep) {\kreis};
\node (l8) at (4*\csep,16*\rsep) {\kreis};
\node (l9) at (4*\csep,15*\rsep) {\kreis};
\node (l10) at (4*\csep,14*\rsep) {\kreis};
\node (l11) at (4*\csep,13*\rsep) {\kreis};
\node (l12) at (4*\csep,12*\rsep) {\kreis};
\node (l13) at (4*\csep,11*\rsep) {\kreis};
\node (l14) at (4*\csep,10*\rsep) {\kreis};
\node (l15) at (4*\csep,9*\rsep) {\kreis};
\node (l16) at (4*\csep,8*\rsep) {\kreis};
\node (l17) at (4*\csep,7*\rsep) {\kreis};
\node (l18) at (4*\csep,6*\rsep) {\kreis};
\node (l19) at (4*\csep,5*\rsep) {\kreis};
\node (l20) at (4*\csep,4*\rsep) {\kreis};
\node (l21) at (4*\csep,3*\rsep) {\kreis};
\node (l22) at (4*\csep,2*\rsep) {\kreis};
\node (l23) at (4*\csep,1*\rsep) {\kreis};
\node (l24) at (4*\csep,0*\rsep) {\kreis};

\draw [->] (w1) -- node [midway,xshift=3*\csep,fill=white]{\lss{avg.}} (l2);
\draw [->] (w1) -- node [above,xshift=-3*\csep,sloped] {\lss{high}} (l1);
\draw [->] (w1) -- node [below,xshift=4*\csep] {\lss{low}} (l3);

\draw [->] (w2) -- node [midway,xshift=3*\csep,fill=white]{\lss{avg.}} (l5);
\draw [->] (w2) -- node [above,xshift=-3*\csep,sloped] {\lss{high}} (l4);
\draw [->] (w2) -- node [below,xshift=4*\csep] {\lss{low}} (l6);

\draw [->] (w3) -- node [midway,xshift=3*\csep,fill=white]{\lss{avg.}} (l8);
\draw [->] (w3) -- node [above,xshift=-3*\csep,sloped] {\lss{high}} (l7);
\draw [->] (w3) -- node [below,xshift=4*\csep] {\lss{low}} (l9);

\draw [->] (w4) -- node [midway,xshift=3*\csep,fill=white]{\lss{avg.}} (l11);
\draw [->] (w4) -- node [above,xshift=-3*\csep,sloped] {\lss{high}} (l10);
\draw [->] (w4) -- node [below,xshift=4*\csep] {\lss{low}} (l12);

\draw [->] (w5) -- node [midway,xshift=3*\csep,fill=white]{\lss{avg.}} (l14);
\draw [->] (w5) -- node [above,xshift=-3*\csep,sloped] {\lss{high}} (l13);
\draw [->] (w5) -- node [below,xshift=4*\csep] {\lss{low}} (l15);

\draw [->] (w6) -- node [midway,xshift=3*\csep,fill=white]{\lss{avg.}} (l17);
\draw [->] (w6) -- node [above,xshift=-3*\csep,sloped] {\lss{high}} (l16);
\draw [->] (w6) -- node [below,xshift=4*\csep] {\lss{low}} (l18);

\draw [->] (w7) -- node [midway,xshift=3*\csep,fill=white]{\lss{avg.}} (l20);
\draw [->] (w7) -- node [above,xshift=-3*\csep,sloped] {\lss{high}} (l19);
\draw [->] (w7) -- node [below,xshift=4*\csep] {\lss{low}} (l21);

\draw [->] (w8) -- node [midway,xshift=3*\csep,fill=white]{\lss{avg.}} (l23);
\draw [->] (w8) -- node [above,xshift=-3*\csep,sloped] {\lss{high}} (l22);
\draw [->] (w8) -- node [below,xshift=4*\csep] {\lss{low}} (l24);
\end{tikzpicture}\label{fig:majormap}
}
\quad
\subfloat[A staged tree $\st\in\ptrep$.]{
\centering
\begin{tikzpicture}
\renewcommand{\csep}{1.7}
\renewcommand{\rsep}{0.23}
\node at (0*\csep,21*\rsep) [text width=1*\csep] {\mbox{access} \mbox{to credit}};
\node at (1*\csep,25*\rsep) [text width=1*\csep] {hospital \mbox{admission}};
\node at (2.6*\csep,25*\rsep) {life events};

\node (v0) at (0*\csep,11*\rsep) {\vertex{0}};

\node (v1) at (1*\csep,20.5*\rsep) {\stage{Red!70}{1}};
\node (v2) at (1*\csep,14.5*\rsep) {\stage{Red!70}{2}};
\node (v3) at (1*\csep,8.5*\rsep) {\stage{Red!70}{3}};
\node (v4) at (1*\csep,4*\rsep) {\stage{Green!70}{4}};
\node (v5) at (1*\csep,1*\rsep) {\stage{Green!70}{5}};

\node (w1) at (2*\csep,22*\rsep) {\stage{SkyBlue}{6}};
\node (w2) at (2*\csep,19*\rsep) {\stage{SkyBlue}{7}};
\node (w3) at (2*\csep,16*\rsep) {\stage{SkyBlue}{8}};
\node (w4) at (2*\csep,13*\rsep) {\stage{Green!70}{9}};
\node (w5) at (2*\csep,10*\rsep) {\stage{SkyBlue}{\text{{\fontsize{1mm}{2mm}\selectfont 10}}}};
\node (w6) at (2*\csep,7*\rsep) {\stage{Green!70}{\text{{\fontsize{1mm}{2mm}\selectfont 11}}}};

\node (l1) at (3*\csep,23*\rsep) {\kreis};
\node (l2) at (3*\csep,22*\rsep) {\kreis};
\node (l3) at (3*\csep,21*\rsep) {\kreis};
\node (l4) at (3*\csep,20*\rsep) {\kreis};
\node (l5) at (3*\csep,19*\rsep) {\kreis};
\node (l6) at (3*\csep,18*\rsep) {\kreis};
\node (l7) at (3*\csep,17*\rsep) {\kreis};
\node (l8) at (3*\csep,16*\rsep) {\kreis};
\node (l9) at (3*\csep,15*\rsep) {\kreis};
\node (l10) at (3*\csep,14*\rsep) {\kreis};
\node (l11) at (3*\csep,13*\rsep) {\kreis};
\node (l12) at (3*\csep,12*\rsep) {\kreis};
\node (l13) at (3*\csep,11*\rsep) {\kreis};
\node (l14) at (3*\csep,10*\rsep) {\kreis};
\node (l15) at (3*\csep,9*\rsep) {\kreis};
\node (l16) at (3*\csep,8*\rsep) {\kreis};
\node (l17) at (3*\csep,7*\rsep) {\kreis};
\node (l18) at (3*\csep,6*\rsep) {\kreis};
\node (l19) at (2*\csep,5*\rsep) {\kreis};
\node (l20) at (2*\csep,4*\rsep) {\kreis};
\node (l21) at (2*\csep,3*\rsep) {\kreis};
\node (l22) at (2*\csep,2*\rsep) {\kreis};
\node (l23) at (2*\csep,1*\rsep) {\kreis};
\node (l24) at (2*\csep,0*\rsep) {\kreis};

\draw [->] 	(v0) -- node [above, sloped] {\ls{$++$}} (v1);
\draw [->] 	(v0) -- node [above, sloped] {\ls{$+-$}} (v2);
\draw [->] 	(v0) -- node [above, sloped] {\ls{$-+$}} (v3);
\draw [->] 	(v0) -- node [above, sloped] {\ls{$--$ yes}} (v4);
\draw [->] 	(v0) -- node [below, sloped] {\ls{$--$ no}} (v5);

\draw [->] 	(v1) -- node [above, sloped] {\ls{yes}} (w1);
\draw [->] 	(v1) -- node [above, sloped] {\ls{no}} (w2);
\draw [->] 	(v2) -- node [above, sloped] {\ls{yes}} (w3);
\draw [->] 	(v2) -- node [above, sloped] {\ls{no}} (w4);
\draw [->] 	(v3) -- node [above, sloped] {\ls{yes}} (w5);
\draw [->] 	(v3) -- node [above, sloped] {\ls{no}} (w6);

\draw [->] (w1) -- node [midway,xshift=3*\csep,fill=white]{\lss{avg.}} (l2);
\draw [->] (w1) -- node [above,xshift=-3*\csep,sloped] {\lss{high}} (l1);
\draw [->] (w1) -- node [below,xshift=4*\csep] {\lss{low}} (l3);

\draw [->] (w2) -- node [midway,xshift=3*\csep,fill=white]{\lss{avg.}} (l5);
\draw [->] (w2) -- node [above,xshift=-3*\csep,sloped] {\lss{high}} (l4);
\draw [->] (w2) -- node [below,xshift=4*\csep] {\lss{low}} (l6);

\draw [->] (w3) -- node [midway,xshift=3*\csep,fill=white]{\lss{avg.}} (l8);
\draw [->] (w3) -- node [above,xshift=-3*\csep,sloped] {\lss{high}} (l7);
\draw [->] (w3) -- node [below,xshift=4*\csep] {\lss{low}} (l9);

\draw [->] (w4) -- node [midway,xshift=3*\csep,fill=white]{\lss{avg.}} (l11);
\draw [->] (w4) -- node [above,xshift=-3*\csep,sloped] {\lss{high}} (l10);
\draw [->] (w4) -- node [below,xshift=4*\csep] {\lss{low}} (l12);

\draw [->] (w5) -- node [midway,xshift=3*\csep,fill=white]{\lss{avg.}} (l14);
\draw [->] (w5) -- node [above,xshift=-3*\csep,sloped] {\lss{high}} (l13);
\draw [->] (w5) -- node [below,xshift=4*\csep] {\lss{low}} (l15);

\draw [->] (w6) -- node [midway,xshift=3*\csep,fill=white]{\lss{avg.}} (l17);
\draw [->] (w6) -- node [above,xshift=-3*\csep,sloped] {\lss{high}} (l16);
\draw [->] (w6) -- node [below,xshift=4*\csep] {\lss{low}} (l18);

\draw [->] (v4) -- node [midway,xshift=3*\csep,fill=white]{\lss{avg.}} (l20);
\draw [->] (v4) -- node [above,xshift=-3*\csep,sloped] {\lss{high}} (l19);
\draw [->] (v4) -- node [below,xshift=4*\csep] {\lss{low}} (l21);

\draw [->] (v5) -- node [midway,xshift=3*\csep,fill=white]{\lss{avg.}} (l23);
\draw [->] (v5) -- node [above,xshift=-3*\csep,sloped] {\lss{high}} (l22);
\draw [->] (v5) -- node [below,xshift=4*\csep] {\lss{low}} (l24);
\end{tikzpicture}\label{fig:majorresize}
}
\caption{Two statistically equivalent staged trees for the CHDS data set.}\label{fig:major}
\end{figure}
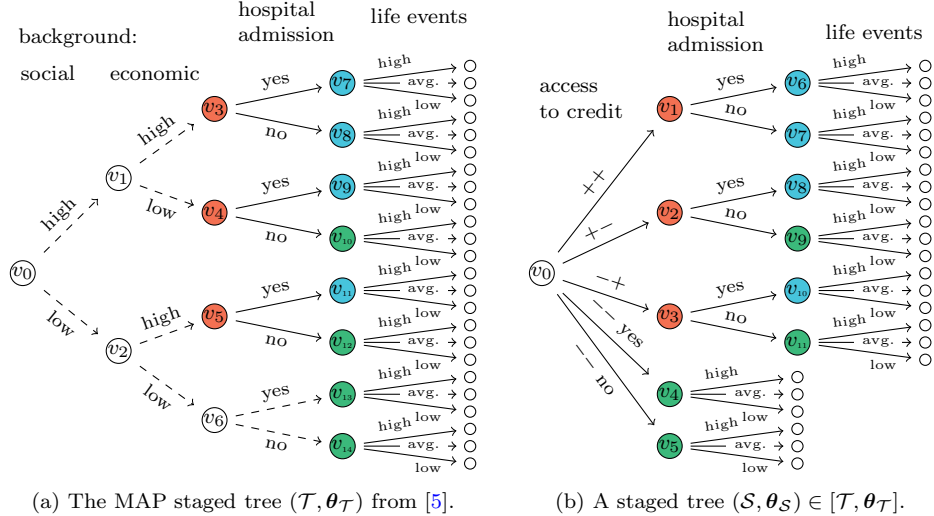

We first observe that there is a saturated subtree in $\pt$, depicted by dotted lines in the figure. Because this does not contain twins, within the polynomial equivalence class there is thus an artificial order on the variables \enquote{social background} and \enquote{economic background}. This order cannot be said to have been deduced from the model search. It is therefore helpful for us to transform $\pt$ into the statistically equivalent staged tree $\st$ of \cref{fig:majorresize}, using a resize operator as in \cref{lem:naiveresize}(a).

The root's edges $e_i=(v_0,v_i)$, $i=1,2,\ldots,5$, in this new tree $\st$ can now be assigned a meaning different from the one in $\pt$. In particular, $e_1$, $e_2$ and $e_3$ correspond to \enquote{social background or economic status are high} and $e_4$ and $e_5$ to \enquote{both social background and economic status are low, hospital admission yes or no}.
Hence, children passing along $e_1$ are \enquote{from a wealthy background}, along $e_2$ and $e_3$ are \enquote{from a moderately wealthy background} and along $e_4$ and $e_5$ are \enquote{from a poor background}. From the stage structure of $\st$ we can see that the probabilities of certain numbers of life events differ between wealthy and poor children. Interestingly, \cite{poverty} names the \emph{access to credit} as a possible monetary measurement of poverty. So being able to borrow from a social network or having own savings is a natural indicator of wealth. This gives some external support for moving from $\pt$ to $\st$, suggested from the results of our automated MAP search on the CHDS data.

We next analyze the polynomial equivalence class $[\St,\Theta_{\St}]^c$ for $c=c_{\St}$. There are five twins in this tree which have two children in the same stage. These are the ones where $v_1$, $v_2\in u_{\text{red}}$, $v_1$, $v_3\in u_{\text{red}}$, $v_2$, $v_3\in u_{\text{red}}$, $v_4$, $v_5\in u_{\text{green}}$ and $v_6$, $v_7\in u_{\text{blue}}$ have the same parent and are in the same stage, respectively. For most of these, an application of the swap operator would be na\"ive and violate the stage constraints in the tree. In fact, there are only two swaps which yield a staged tree: the composition which performs a floret-swap on $v_1$, $v_2$ and $v_3$ simultaneously and the swap on $v_6$, $v_7$. These change the order between access to credit and hospital admission for (moderately) wealthy children, and between hospitalisation and life events for wealthy children. It would thus be spurious to assert a potentially causal or chronological order on these events on the basis of the MAP search.

There is no staged tree in the polynomial equivalence class of $\st$ that would allow for the total order \enquote{life events before hospitalisation}. This is because no composition of the swaps on the twins can form a level-swap on $\st$. So a model which treats life events as an explanatory variable of the response variable hospital admission as in the study~\cite{lorna:bnCEG} is less supported by the data than one treating hospitalisation as an explanatory variable of life events as in~\cite{cowelljim:causal}. Note that no deductions about an ordering of variables would have been possible within the original BN representation of the data because the MAP model turns out to be decomposable. This demonstrates that the extra structure of the staged tree enables us to draw out new potential causal hypotheses that could not be discovered when using more conventional graphical methods.

\section{Discussion}\label{sect:broader}
In this paper we have been able to show that a characterization of staged trees in terms of their interpolating polynomials provides an elegant way to fully analyze statistical equivalence classes of models represented by such trees. 

For a future implementation of these results it is important to note that the number of tree-compatible factorizations of an interpolating polynomial is usually enormous. For instance, a na\"ive count of the elements in the class analyzed in \cref{sect:causal} reveils nearly one thousand elements. This is because every polynomial equivalence class contains $2^{\#\text{twins}}$ elements arising from na\"ive swaps, each of which can be combined with resize operations as outlined above. Of course we would then need to check how many of these elements actually correspond to staged trees. This will normally reduce the number of amenable tree-compatible factorizations significantly: for instance in that example there would be $2^5=32$ na\"ive representations, only four of which are staged.

The polynomial-based approach we develop here provides a most promising foundation for developing an efficient search across this class using computational algebra. In particular, given any discrete model with multilinear parametrization, we note that every possible tree-compatible factorization of its interpolating polynomial arises from a certain nested order of common divisors of terms in the polynomial. This nesting is naturally reflected in what is called the \emph{primary decomposition} of the ideal spanned by all terms in the polynomial. Every element of such a decomposition which is spanned by degree-one indeterminates will then provide a set of putative root labels of a corresponding tree representation; and if there are no such candidates then the multilinear model would not be a staged tree model. Investigating subnestings of ideals and running a search over putative root-labels obtained from ideal decomposition is much faster than for instance an exhaustive search over all possible nested factorizations of a polynomial. This is because in such an unstructured search we would have in the order of $2^d$ choices of root labels, one for each subset of labels, where $d$ is the number of indeterminates. Ideal decomposition provides us with much fewer candidate nestings and also provides an elegant way of a priori excluding certain na\"ive representations which are not staged. As a consequence, we can employ a vast range of freely available software to design algorithms which can efficiently traverse a polynomial equivalence class---so an algorithmic implementation for the swap operator is within reach. As for the resize operator, the computational-algebra algorithm suggested here will then need to be enhanced by a command which allows us to leave a fixed algebraic framework (given by the chosen parametrization) and to substitute terms using the requirements for having non-na\"ive resizes we discovered in this paper. The design of these algorithms is outside the scope of this publication.

In a second step, once an algorithm as sketched above is in place and we have software for applying swaps and resizes on a staged tree, we are ready to use these results in inference and model selection. In particular, when data is available we can now develop methods to score an interpolating polynomial (rather than a tree graph) directly and use the methods proposed here to then traverse the whole statistical equivalence class purely algebraically.
An important direction for future work is to demonstrate how interpolating polynomials can thus be used in the analysis of tree-based causality, in comparison to analogous concepts developed for BN models.

\section*{Acknowledgements}
Christiane G\"orgen was supported by the EPSRC grant EP/L505110/1 and Jim Q.~Smith was supported by the EPSRC grant EP/K039628/1.

\bibliographystyle{imsart-nameyear}
\bibliography{bj_goergensmith_bib.bib}
\end{document}